\documentclass{basm}

\usepackage{calc}

\usepackage{tikz}

\usetikzlibrary{decorations.markings}

\tikzstyle{vertex}=[circle, draw, inner sep=0pt, minimum size=6pt]

\newcommand{\vertex}{\node[vertex]}

\usepackage{amsmath,amssymb,url}
\usepackage{lineno}
\usepackage{verbatim}

\numberwithin{equation}{section}

\DeclareMathOperator{\Lh}{Lh}
\DeclareMathOperator{\Rh}{Rh}
\DeclareMathOperator{\Lbranch}{Lbranch}
\DeclareMathOperator{\Rbranch}{Rbranch}

\newcommand{\eq}[1]{\begin{equation} #1 \end{equation}}
\newcommand{\eql}[2]{\begin{equation} \label{#1} #2 \end{equation}}
\newcommand{\eqt}[2]{\begin{equation*} #2 \tag{#1} \end{equation*}}

\newcommand{\ga}{\alpha}
\newcommand{\gb}{\beta}
\renewcommand{\gg}{\gamma}
\newcommand{\gd}{\delta}

\newcommand{\gl}{\lambda}

\newcommand{\gr}{\varrho}
\newcommand{\gs}{\sigma}

\newcommand{\gt}{\tau}

 \author{Amir Ehsani, Aleksandar Krape\v{z} and Yuri Movsisyan}

\title{Algebras with Parastrophically Uncancellable Quasigroup Equations}

\RunningHead{A. Ehsani, A. Krape\v{z}, Yu. Movsisyan}%
            {Algebras With Parastrophically Uncancellable Quasigroup Eqs.}



\Abstract{We consider $48$ parastrophically uncancellable quadratic functional equations with four object variables and two quasigroup operations in two classes: balanced non--Belousov (consists of $16$ equations) and non--balanced 
non--gemini (consists of $32$ equations). A linear representation of a group (Abelian group) for a pair of quasigroup operations satisfying one of these parastrophically uncancellable quadratic equations is obtained. As a consequence of these results, a linear representation for every operation of a binary algebra satisfying one of these hyperidentities is obtained.}

\CopyRight { Amir Ehsani, Aleksandar Krape\v{z} and Yuri Movsisyan}

\Address
{{\sc Amir Ehsani}\\
{Department of Mathematics, Mahshahr Branch, Islamic Azad University, Mahshahr, Iran.}\\
E-mail: \emph{a.ehsani@mhriau.ac.ir}
\medskip

{\sc Aleksandar Krape\v{z}}\\
{Mathematical Institute of the Serbian Academy of Sciences and Arts, Knez Mihailova 36, 11001 Belgrade, Serbia.}\\
E-mail: \emph{sasa@mi.sanu.ac.rs}
\medskip

{\sc Yuri Movsisyan}\\
{Department of Mathematics and Mechanics,  Yerevan State University,  Alex Manoogian 1, Yerevan 0025,
Armenia.}\\
E-mail: \emph{yurimovsisyan@yahoo.com}
}

 \Received { \ }

\begin{document}

\maketitle

{\flushright \bf Dedicated to \\
V. D. Belousov \\
and \\
G. B. Belyavskaya \\}
%%%%%%%%%%%%%%%%%%%%%%%%%%%%%%%%%%%%%%%%%%%%%%%%%%%%%%%%%%%%%%%
\section{Introduction}
%\linenumbers

A binary quasigroup is usually defined to be a groupoid
$(B; f)$ such that for any $a,b\in B$ there are unique solutions
$x$ and $y$ to the following equations:
 \[f(a,x)=b \ \ \ \text{and}\ \ \ f(y,a)=b.\]
The basic properties of quasigroups were given in books  \cite{26, Denes, pfl, CheinPflugfelderSmith}. We remind the reader of those properties we shall use in the paper.

If $(B; f)$ is quasigroup we say that $f$ is a quasigroup operation.
A loop is a quasigroup with unit $(e)$ such that
\[f(e,x)=f(x,e)=x.\] Groups are associative quasigroups, i.e. they
satisfy:
\[f(f(x,y),z)=f(x,f(y,z))\] and they necessarily contain a unit.
A quasigroup is commutative if 
\begin{equation}\label{comm}
f(x,y)=f(y,x).
\end{equation}
Commutative groups are also known as Abelian groups.

A triple $(\alpha,\beta,\gamma)$ of bijections from a set $B$ onto
a set $C$ is called an isotopy of a groupoid $(B; f)$ onto a
groupoid $(C; g)$ provided \[\gamma f(x, y)=g(\alpha x,\beta
y)\] for all $x,y\in B$. 
$(C; g)$ is then called an isotope of
$(B; f)$, and groupoids $(B; f)$ and $(C; g)$ are
called isotopic to each other. An isotopy of $(B; f)$ onto
$(B; f)$ is called an autotopy of $(B; f)$.
Let $\alpha$ and $\beta$ be permutations of $B$ and let $\iota$ denote
the identity map on $B$. Then $(\alpha,\beta,\iota)$ is a
principal isotopy of a groupoid $(B; f)$ onto a groupoid
$(B; g)$ means that $(\alpha,\beta,\iota)$ is an isotopy of
$(B; f)$ onto  $(B; g)$.
Isotopy is a generalization of isomorphism. Isotopic image of a
quasigroup is again a quasigroup. A loop isotopic to a group is isomorphic to it.
Every quasigroup is isotopic to some loop i.e., it is a loop isotope.

If $(B; +)$ is a group, then the bijection $\alpha:
B\rightarrow B$ is called a \textit{holomorphism} of $(B;+)$ if
\eq{\alpha (x+ y^{-1}+ z)=\alpha x+ (\alpha y)^{-1}+ \alpha z.}
%$\alpha$ is called an \textit{antiholomorphism} if \[
%\alpha (x+ y^{-1}+ z)=\alpha z+ (\alpha y)^{-1}+ \alpha
%x,\] for every $x,y,z\in B$. 
The set of all holomorphisms  of $(B;+)$ is denoted by
$Hol(B; +)$. It is a group under the composition of
mappings: $(\alpha \cdot\beta)x=\beta(\alpha x)$, for every $x\in
B$. Note that this concept is equivalent to the concept of quasiautomorphism of groups, by \cite{26}. 

\begin{comment}
Let ${\bf A}=(A; F)$ and ${\bf B}=(B; F)$ be two binary algebras of the same type. Then a mapping
$\alpha :A\rightarrow B$  is called \textit{antihomomorphism} of
{\bf A} into {\bf B}, if for every
$f\in F$ we have:\[\alpha f^{\bf A}(x_1,x_2)=f^{\bf B}(\alpha
x_2,\alpha x_1).\]
In the case ${\bf A}={\bf B}$, an antihomomorphism is called an
\textit{antiendomorphism} and if $\alpha$ is a bijection, then it is called
an \textit{antiautomorphism}. 
The set of all antiautomorphisms defined on an algebra ${\bf A}=(A; F)$ is denoted by $Antiaut(A; F)$.
\end{comment}

A binary quasigroup $(B;f)$ is linear over a  group (Abelian group) if
\[f(x,y)=\varphi x+a+\psi y,\] where $(B;+)$ is a  group (Abelian group), $\varphi$
and $\psi$ are automorphisms of  $(B;+)$ and $a\in B$ is
a fixed element.
  Quasigroup linear over an Abelian group is also called a
 $T$-quasigroup.

Quasigroups are important algebraic (combinatorial, geometric) structures which
arise in various areas of mathematics and other disciplines. We mention just a few
of their applications: in combinatorics (as latin squares, see \cite{Denes}), in geometry (as nets/webs, see \cite{Belousov}), in statistics (see \cite{Fisher}), in special theory of relativity (see \cite{Ungar}), in coding theory and cryptography (\cite{Shcher}).

%%%%%%%%%%%%%%%%%%%%%%%%%%%%%%%%%%%%%%%%%%%%%%%%%%%%%%%%%%%%%%%%%%%%%%
\section{Preliminaries}

We use (object) variables $x, y, z, u, v, w$ (perhaps with indices) and operation symbols (i.e. functional variables) $f, g, h$  (also with indices).
We assume that all operation symbols represent quasigroup operations.

The set of all variables which appear in a term $t$ is called the \textit{content} of $t$ and is denoted by $var(t)$. 
 A variable $x$ is \textit{linear variable} in a term $t$, when it occurs just once in $t$. A variable $x$ is \textit{quadratic variable} in a term $t$, when it occurs twice in  $t$. The sets of all linear and quadratic variables of term $t$ are denoted by $var_1(t)$ and $var_2(t)$, respectively.

A \textit{functional equation} is an equality $s = t$, where $s$ and $t$ are terms with symbols of unknown operations occurring in at least one of them.
 
%\begin{defn}
%We say that a subterm $p$ of the equation $s = t$
%is a \textit{closed subterm},  if every variable from $p$ occurs twice
%in $p$.
%\end{defn}
%
\begin{defn}
A functional equation $s=t$ is \textit{quadratic} if every (object) variable occurs exactly twice in $s=t$.
%Functional equation $s=t$
It is \textit{balanced} if every (object) variable appears exactly once in $s$ and once in $t$. 
\end{defn}
\begin{defn}
A variable $x$ from a quadratic equation $s=t$ is \emph{linear} if $x$ occurrs once in $s$ and once in $t$\/; it is \emph{left (right) quadratic} if it occurrs twice in $s \, (t)$ and \emph{quadratic} if it is either left or right quadratic.
\end{defn}
\begin{defn}
A balanced equation $s = t$ is \textit{Belousov} if for every subterm $p$ of $s$ ($t$) there is a subterm $q$ of $t$ ($s$) such that $p$ and $q$ have exactly the same variables.
\end{defn}
\begin{defn}
A quadratic quasigroup equation is \textit{gemini} iff it is a theorem of $TS$-loops ($=$ Steiner loops) i.e., consequence of the identities of the variety of $TS$-loops.
\end{defn}
\begin{defn}
Functional equation $s = t$ is \textit{generalized} if every operation symbol from $s = t$ occurrs there just once.
\end{defn}
\begin{defn}
Let $x$ be a variable occurring in a quadratic equation $s = t$\/. The function $\Lh \; (\Rh)$ of the \em{left (right) height} of the variable $x$ in the equation $s = t$ is given by: 
\begin{description}
\item[$-$] If $x\notin var(t)$, then $\Lh(x,t) \;\; (\Rh(x,t))$ is not defined,
\item[$-$] $Lh(x,x)= 0 \;\; (Rh(x,x)= 0)$\/,
\item[$-$] If $t = f(t_1,t_2)$ and there is no occurrence of $x$ in $t_2$ then 
               $\Lh(x,t) = 1 + \Lh(x,t_1) \;\; (\Rh(x,t) = 1 + \Rh(x,t_1))$\/,
\item[$-$] If $t = f(t_1,t_2)$ and there is no occurrence of $x$ in $t_1$ then 
               $\Lh(x,t) = 1 + \Lh(x,t_2) \;\; (\Rh(x,t) = 1 + \Rh(x,t_2))$\/,
\item[$-$] If $t = f(t_1,t_2)$ and $x$ occurrs in both $t_1$ and $t_2$ then 
               $\Lh(x,t) = 1 + \Lh(x,t_1) \;\; (\Rh(x,t) = 1 + \Rh(x,t_2))$\/,
\item[$-$] $
\Lh(x,s=t) = 
\begin{cases}
\Lh(x,s),  & \text{if } x \in var(s) \\
\Lh(x,t), & \text{otherwise,}
 \end{cases}
$
\item[$-$] $
\Rh(x,s=t) = 
\begin{cases}
\Rh(x,t),  & \text{if }x \in var(t) \\
\Rh(x,s), & \text{otherwise.}
 \end{cases}
$
\end{description}
\end{defn}

\begin{defn}
Let $s = t$ be a quadratic equation. It is a \em{level equation} iff $\Lh(x,s = t) = \Rh(y,s = t)$ for all variables $x, y$ of $s = t$\/. 
\end{defn}

 \begin{exa}
The following are various functional equations:
\begin{align}
\label{com}&\text{(commutativity)} \qquad &f(x,y)=f(y,x) , \\
\label{ass}&\text{(associativity)} \qquad  &f(f(x,y),z)=f(x,f(y,z)) , \\
\label{med}&\text{(mediality)} \qquad &f(f(x,y),f(u,v))=f(f(x,u),f(y,v)) ,\\
\label{par} &\text{(paramediality)} \qquad &f(f(x,y),f(u,v))=f(f(v,y),f(u,x)) ,\\
%\label{pse}&\text{(pseudomediality)} \qquad &f(f(x,y),f(u,v))=f(f(f(x,u),y),v),  \\
\label{dis}&\text{(distributivity)} \qquad  &f(x,f(y,z))=f(f(x,y),f(x,z)) , \\
\label{tra}&\text{(transitivity)} \qquad  &f(f(x,y),f(y,z))=f(x,z) ,\\
\label{inmed} & \text{(intermediality)} \qquad &f(f(x,y),f(y,u))=f(f(x,v),f(v,u)) ,\\
\label{exmed} & \text{(extramediality)}\qquad  &f(f(x,y),f(u,x))=f(f(v,y),f(u,v)) ,\\
\label{4pal} & \text{(4-palindromic identity)}\qquad & f(f(x,y),f(u,v))=f(f(v,u),f(y,x)) ,\\
\label{ide}&\text{(idempotency)} \qquad &f(x,x)=x  ,\\
\label{trivial} &\text{(trivial)} \qquad & f(x,y)=f(x,y) ,\\
\label{00} &&f(x,f(y,z))=f(f(z,y),x). 
\end{align}
%Only idempotency has a closed subterm, the rest of equations does not have any closed subterms. 
Associativity, (para)mediality, 
%pseudomediality, 
$4$-palindromic, trivial identity and  $\eqref{00}$ are balanced, 
transitivity, intermediality and extramediality  are quadratic but not balanced and idempotency and (left) distributivity are not even quadratic.
 Commutativity, trivial, $4$-palindromic and $\eqref{00}$ are gemini functional equations and since they are balanced,  they  are Belousov equations as well. The equations $\eqref{ass}-\eqref{exmed}$ are non-gemini and non-Belousov.
Commutativity, mediality, paramediality, intermediality, extramediality, 4--palindromic and trivial identity are level equations. 
 \end{exa}

Every quasigroup satisfying (para)medial identity is called \textit{(para)medial quasigroup}. Every quasigroup satisfying $4$-palindromic identity is called \textit{$4$-palindromic quasigroup}.

\begin{thm}[Toyoda \cite{22}]\label{quasimedial}
If $(B; f)$ is a medial quasigroup then there exists an Abelian group $(B;+)$, such that
$f(x, y)=\varphi(x)+c+\psi(y)$, where $\varphi,\psi\in Aut
(B; +)$, $\varphi\psi=\psi\varphi$ and $c\in B$.
\end{thm}

\begin{thm}[N\v{e}mec, Kepka \cite{13}]\label{quasipara}
 If $(B; f)$ is a paramedial quasigroup then there exists an Abelian group $(B; +)$, such that
$f(x, y)=\varphi(x)+c+\psi(y)$, where $\varphi,\psi\in Aut(B; +)$, $\varphi\varphi=\psi\psi$ and $c\in B$.
\end{thm}

More generaly, considering the following equations with two functional variables, we can define the notion of (para)medial pair of operations:
\begin{align}
\label{medp}f_1(f_2(x,y),f_2(u,v))=f_2(f_1(x,u),f_1(y,v)) ,\\
\label{parp} f_1(f_2(x,y),f_2(u,v))=f_2(f_1(v,y),f_1(u,x)).
\end{align} 

\begin{defn}
A pair $(f_1, f_2)$ of binary operations is called \textit{(para)medial pair of operations}, if the algebra $(B; f_1, f_2)$ satisfies the equation $\eqref{medp}$ ($\eqref{parp}$).
    \end{defn}

\begin{defn}
A binary algebra $\textbf{B}=(B; F)$ is called \textit{(para)medial algebra}, if every pair of operations of the algebra $\textbf{B}$ is (para)medial (or, the algebra $\textbf{B}$ satisfies (para)medial hyperidentity).
    \end{defn}

The following theorem generalizes above results by Toyoda and N\v{e}mec, Kepka:
\begin{thm}[Nazari, Movsisyan \cite{nazari}, Ehsani, Movsisyan \cite{ehsanimovsisyan}]\label{medial}
Let the set $B$, forms a quasigroup under the binary operations
$f_1$ and $f_2$. If the  pair of binary operations $(f_1, f_2)$ is
(para)medial, then there exists a binary operation $'+'$ under which $B$
forms an Abelian group and for arbitrary elements $x,y \in B$ we
have:
\[f_i(x,y)=\varphi_i(x)+\psi_i(y)+c_i,\]
where $c_i$s are fixed elements of $B$, and $\varphi_i, \psi_i \in Aut(B;+)$ for
$i=1,2$, such that: 
\begin{description}
\item
$\varphi_1 \psi_2=\psi_2\varphi_1$, $ \varphi_2\psi_1=\psi_1\varphi_2$, $ \psi_1\psi_2=\psi_2\psi_1$ and $ \ \varphi_1\varphi_2=\varphi_2\varphi_1$ should be satisfied by the medial pair of operations,
\item
$\varphi_1\varphi_2=\psi_2\psi_1$, $ \varphi_2\varphi_1=\psi_1\psi_2$, $ \varphi_1\psi_2=\varphi_2\psi_1$ and $\ \psi_1\varphi_2=\psi_2\varphi_1$ should be satisfied by the paramedial pair of operations.
\end{description}
The group $(B;+)$, is unique up to isomorphisms.
\end{thm}

The following results will be frequently utilized.
\begin{thm}[Acz\'{e}l, Belousov, Hossz\'{u} \cite{1}, see also \cite{Belousov65}]\label{ABH}
Let the set $B$ forms a quasigroup under six operations $A_i(x,y)$
(for $i=1,\dots,6$). If these operations satisfy the following equation:
\begin{equation}\label{t1}
A_1(A_2(x,y),A_3(u,v))=A_4(A_5(x,u),A_6(y,v)),
\end{equation}
for all elements $x$, $y$, $u$ and $v$ of the set $B$ then there exists an
operation '$+$' under which $B$ forms an abelian group isotopic to all these six quasigroups. And there exist eight
permutations $\alpha$, $\beta$, $\gamma$, $\delta$, $\epsilon$,
$\psi$, $\varphi$, $\chi$ of $B$ such that:
\begin{eqnarray*}
&&A_1(x,y)=\delta x+\varphi y,\\
&&A_2(x,y)=\delta^{-1}(\alpha x+\beta y),\\
&&A_3(x,y)=\varphi^{-1}(\chi x+\gamma y),\\
&&A_4(x,y)=\psi x+\epsilon y,\\
&&A_5(x,y)=\psi^{-1}(\alpha x+\chi y),\\
&&A_6(x,y)=\epsilon^{-1}(\beta x+\gamma y).
\end{eqnarray*}
\end{thm}

\begin{thm}[Krape\v{z}, \cite{krapez1981}]\label{krapez theorem}
If the set $B$ forms a quasigroup under four operations $A_i(x,y)$ (for $i=1,\ldots,4$) and if these operations satisfy the equation of  generalized transitivity:
\[A_1(A_2(x,y),A_3(y,z))=A_4(x,z),\]
for all elements $x, y, z \in B$, then there exists an operation $'+'$ under which $B$ forms a group isotopic to all these  quasigroups and there exist permutaions $\alpha$, $\beta$, $\gamma$, $\delta$, $\epsilon$, $\psi$, $\varphi$, $\chi$ of $B$ such that
\begin{align*}
&A_1(x,y)=\alpha x+\beta y,\\
&A_2(x,y)=\alpha^{-1}(\alpha \gamma x+\alpha \delta y),\\
&A_3(x,y)=\beta^{-1}(\beta \epsilon x+\beta \psi y),\\
&A_4(x,y)=\varphi x+\chi y.
\end{align*}
\end{thm}

\begin{thm}[Krape\v{z} \cite{krapez 79}, Belousov \cite{Belousov85}]\label{Krap,Bel}
A quasigroup satisfying a balanced but not Belousov equation is isotopic to a group.
\end{thm}

\begin{thm}[Krape\v{z}, Taylor \cite{krapezTaylor95}]\label{KrapTay}
A quasigroup satisfying a quadratic but not gemini equation is isotopic to a group.
\end{thm}

%%%%%%%%%%%%%%%%%%%%%%%%%%%%%%%%%%%%%%%%%%%%%%%%%%%%%%%%%%%%%%%
\section{Parastrophically uncancellable quadratic equations with two function variables}

We consider parastrophically uncancellable quadratic quasigroup equations of the form:
%\begin{equation}\label{300}
\eqt{Eq}{f_1(f_2(x_1,x_2),f_2(x_3,x_4))=f_2(f_1(x_5,x_6),f_1(x_7,x_8))}
%\end{equation}
where $x_i\in \{x,y,u,v\}$, for $i=1,\ldots,8$. 
Therefore, the equation (Eq) is quadratic level quasigroup equation with four (object) variables each appearing twice in the equation and with two function variables each appearing three times in the equation. 
%Moreover equation is chosen to be non-gemini. 
There are $48$ such equations and we attempt to solve them all.

There is a correspondence between generalized quadratic quasigroup equations and connected cubic graphs, namely Krsti\'{c} graphs. 
Two such equations are parastrophically equivalent  iff they have the same (i.e. isomorphic) Krsti\'{c} graphs. 
Furthermore, an equation is parastrophically uncancellable iff the corresponding Krsti\'c graph is 3--connected. For more detailed account of this correspondence see \cite{krapezZivkovic}, \cite{krapezTaylor95} and \cite{KrsticDr}.

For every one of the $48$ equations (Eq) there is a corresponding generalized equation:
%\begin{equation}\label{301}
\eqt{GEq}{f_1(f_3(x_1,x_2),f_4(x_3,x_4))=f_2(f_5(x_5,x_6),f_6(x_7,x_8))}
%\end{equation}
(where $x_i\in \{x,y,u,v\}$, for $i=1,\ldots,8$)
with the appropriate Krsti\'{c} graph. This Krsti\'{c} graph \emph{will be assumed} to be the Krsti\'{c} graph of (Eq) as well.
All these equations can be partitioned into two classes, depending on their Krsti\'{c} graphs, as follows:

\begin{description}
\item[-] $16$ balanced (and non-Belousov) equations with the Krsti\'{c} graph $K_{3,3}$,
\item[-] $32$ non--balanced non-gemini equations with the Krsti\'{c} graph $P_{3}$.
\end{description}

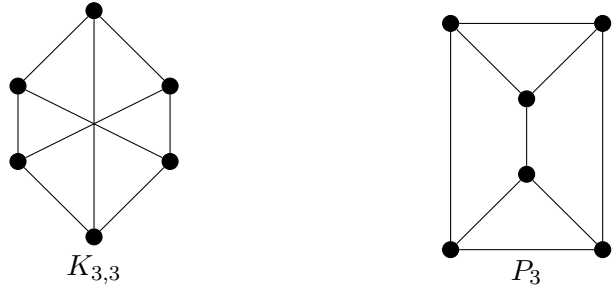
\begin{center}
$\begin{tikzpicture}
	\vertex[fill] (1) at (1, 3) {};
	\vertex[fill] (2) at (0, 2) {};
	\vertex[fill] (3) at (0, 1) {};
	\vertex[fill] (4) at (1, 0) [label=below:$K_{3,3}$]{};
	\vertex[fill] (5) at (2, 1) {};
	\vertex[fill] (6) at (2, 2) {};
\path 
 	(1) edge (2)
	(1) edge (6)
	(1) edge (4) 		
	(2) edge (3)
 	(2) edge (5)
 	(3) edge (4)
	(3) edge (6)
 	(4) edge (5)
 	(5) edge (6)
	;
\end{tikzpicture}  \qquad \qquad$ \ \  \ \ 
 \ \  \ \ 
$ \qquad \begin{tikzpicture}
	\vertex[fill] (1) at (2, 3) {};
	\vertex[fill] (2) at (0, 3) {};
	\vertex[fill] (3) at (0, 0) {};
	\vertex[fill] (4) at (2, 0) {};
	\vertex[fill] (5) at (1, 2) {};
	\vertex[fill] (6) at (1, 1) {};
\path 
	(1) edge (2) 	
	(1) edge (5)
	(1) edge (4)
	(2) edge (3)
 	(2) edge (5)
 	(3) edge node[below]{$P_{3}$}(4) 
	(3) edge (6)
 	(4) edge (6)
 	(5) edge (6)
	;
\end{tikzpicture}$
\end{center}

To characterize a pair of quasigroup operations which satisfies a non--Belousov balanced functional equation, we need the notion of $\Lbranch \; (\Rbranch)$ and the following properties of holomorphisms which were proved for Muofang loops in \cite{10}.

\begin{defn}
\label{branches}
Let $t$ be a term and $x$ a variable. We define:
\begin{itemize}
\item If $x\notin var(t)$, then $\Lbranch(x,t) \;\; (\Rbranch(x,t))$ is not defined,
\item $\Lbranch(x,x)= \Lambda \;\; (\Rbranch(x,x)= \Lambda)$ \; ($\Lambda$ is the empty word),
\item If $t=f_i(t_1, t_2)$ and there is no occurrence of $x$ in $t_2$ then
        $\Lbranch(x,t)=\alpha_i \Lbranch(x,t_1) \;\; (\Rbranch(x,t)=\alpha_i \Rbranch(x,t_1))$\/,
\item If $t=f_i(t_1, t_2)$ and there is no occurrence of $x$ in $t_1$ then
        $\Lbranch(x,t)=\gb_i \Lbranch(x,t_2) \;\; (\Rbranch(x,t)=\gb_i \Rbranch(x,t_2))$\/,
\item If $t=f_i(t_1, t_2)$ and $x$ occurrs in both $t_1$ and $t_2$ then
        $\Lbranch(x,t)=\alpha_i \Lbranch(x,t_1) \;\; (\Rbranch(x,t)=\beta_i \Rbranch(x,t_2))$\/,
\item $
        \Lbranch(x,s=t) = 
        \begin{cases}
        \Lbranch(x,s),  & \text{if } x \in var(s) \\
        \Lbranch(x,t), & \text{otherwise,}
        \end{cases}
        $\/,
\item $
        \Rbranch(x,s=t) = 
        \begin{cases}
        \Rbranch(x,t),  & \text{if } x \in var(t) \\
        \Rbranch(x,s), & \text{otherwise,}
        \end{cases}
        $\/.
\end{itemize}
\end{defn}
%
\begin{comment}
\begin{defn}
Let $(B; F=\{f_i\}_{i\in I})$ be a binary algebra with quasigroup operations such that every operation $f_i\in F$ is linear over a group (Abelian group) $(B; +)$ (i.e. $f_i(x,y)=\alpha_ix+c_i+\beta_iy$, for every $i\in I$). If $t$ be a non empty term (in the language $F=\{f_i\}_{i\in I}$), then we define the notation of $branch(x,t)$ for every variable $x$, as follows:
\begin{description}
\item[$-$] If $x\notin var(t)$, then $branch(x,t)$ is not defined,
\item[$-$] $branch(x,x)= \Lambda$ \; ($\Lambda$ is the empty word),
\item[$-$] $branch(x,t)=\alpha_i branch(x,t_1)$, if $x\in var(t_1)$ and $t=f_i(t_1, t_2)$ %for every $i\in I$,
\item[$-$] $branch(x,t)=\beta_i branch(x,t_2)$, if $x\in var(t_2)$ and $t=f_i(t_1, t_2)$ %for every $i\in I$,
\item[$-$] $branch(x,t)=\alpha_i branch(x,t_1)+\beta_i branch(x,t_2)$, if $x\in var(t_1)$, $x\in var(t_2)$ and $t=f_i(t_1, t_2)$\/. %for every $i\in I$.
\end{description}
\end{defn}
\end{comment}
%
\begin{lem}\label{hol1}
Let the identity:
\[
\alpha_1(x+ y)=\alpha_2(x)+\alpha_3(y).
\]
be satisfied for bijections $\alpha_1, \alpha_2, \alpha_3$ on the group  $(B;+)$\/.
Then $\alpha_1, \alpha_2,
 \alpha_3\in Hol(B;+)$.
\end{lem}
\begin{lem}\label{hol2}
Every holomorphism $\alpha$ of the group $(B;+)$  has the
following forms:
\[
\alpha x=\varphi_1 x+ k_1, \ \ \ \ \  \alpha x=k_2+\varphi_2 x,
\]
where $\varphi_1, \varphi_2\in Aut(B;+)$ and $k_1, k_2\in B$.
\end{lem}

%
%\newpage
\section{Equations with Krsti\'c graph $K_{3,3}$}

The class of non--gemini balanced (and therefore non--Belousov) quadratic functional equations consists of the following $16$  equations with four object variables $x,y,u,v$ and two quasigroup operations $f_1$, $f_2$:
\begin{align}
\label{2-1}f_1(f_2(x,y),f_2(u,v))=f_2(f_1(x,u),f_1(y,v))\\
\label{2-2}f_1(f_2(x,y),f_2(u,v))=f_2(f_1(x,u),f_1(v,y))\\
\label{2-3}f_1(f_2(x,y),f_2(u,v))=f_2(f_1(x,v),f_1(y,u))\\
\label{2-4}f_1(f_2(x,y),f_2(u,v))=f_2(f_1(x,v),f_1(u,y))\\
\label{2-5}f_1(f_2(x,y),f_2(u,v))=f_2(f_1(y,u),f_1(x,v))\\
\label{2-6}f_1(f_2(x,y),f_2(u,v))=f_2(f_1(y,u),f_1(v,x))\\
\label{2-7}f_1(f_2(x,y),f_2(u,v))=f_2(f_1(y,v),f_1(x,u))\\
\label{2-8}f_1(f_2(x,y),f_2(u,v))=f_2(f_1(y,v),f_1(u,x))\\
\label{2-9}f_1(f_2(x,y),f_2(u,v))=f_2(f_1(u,x),f_1(y,v))\\
\label{2-10}f_1(f_2(x,y),f_2(u,v))=f_2(f_1(u,x),f_1(v,y))\\
\label{2-11}f_1(f_2(x,y),f_2(u,v))=f_2(f_1(u,y),f_1(x,v))\\
\label{2-12}f_1(f_2(x,y),f_2(u,v))=f_2(f_1(u,y),f_1(v,x))\\
\label{2-13}f_1(f_2(x,y),f_2(u,v))=f_2(f_1(v,x),f_1(y,u))\\
\label{2-14}f_1(f_2(x,y),f_2(u,v))=f_2(f_1(v,x),f_1(u,y))\\
\label{2-15}f_1(f_2(x,y),f_2(u,v))=f_2(f_1(v,y),f_1(x,u))\\
\label{2-16}f_1(f_2(x,y),f_2(u,v))=f_2(f_1(v,y),f_1(u,x))
\end{align}

The following result generalizes, on the one hand the Theorem \ref{medial}, and on the other, the results from and immediately after the Example 7 in \cite{krapez2005}.
\begin{thm}
\label{main1}
Let the balanced non--Belousov quasigroup equations \emph{(4.j)  (j = 1,\dots,16)} have the Krsti\'c graph $K_{3,3}$\/.
A general solution of any of \emph{(4.j)} is given by:
\eql{sol4}{f_i(x,y) = {\ga}_i x + c_i + {\gb}_i y \quad (i = 1,2)}
where:
\begin{itemize}
\item $(B;+)$ is an arbitrary Abelian group,
\item $c_1, c_2$ are arbitrary elements of $B$ such that $f_1(c_2,c_2) = f_2(c_1,c_1)$\/,
\item ${\ga}_i, {\gb}_i \; (i = 1,2)$ are arbitrary automorphisms of + such that:
        \eql{lin4}{\Lbranch(z, \emph{(4.j)}) = \Rbranch(z, \emph{(4.j)})}
        for all variables $z$ of the equation \emph{(4.j)}.
\end{itemize}
The group $(B;+)$ is unique up to isomorphism.
\end{thm}
\begin{proof}
(1) 
To show that the pair $(f_1,f_2)$ of operations is a solution of (4.j), just replace $f_i(x,y)$ in (4.j) using (\ref{sol4}) and all conditions (\ref{lin4}).

(2) 
An equation (4.j) is an instance of the appropriate generalized equation (GEq) with the Krsti\'c graph $K_{3,3}$\/.
Therefore, all operations of (GEq) are isotopic to an Abelian group $+$ and the main operations $f_1,f_2$ can be chosen to be principally isotopic to it (see \cite{krapezZivkovic}):
$${f_i(x,y) = {\gl}_i x + {\gr}_i y \quad (i = 1,2).}$$
Replace this in (Eq) to get:
\eql{cons}{{\gl}_1 f_2(x_1,x_2) + {\gr}_1 f_2(x_3,x_4) = {\gl}_2 f_1(x_5,x_6) + {\gr}_2 f_1(x_7,x_8).}
Since variables $x_1,x_2$ are separated on the right hand side of the equation (\ref{cons}), replacing $x_3$ and $x_4$ by $0$\/, we get:
$${\gl}_1 ({\gl}_2 x_1 + {\gr}_2 x_2) + d = \gs x_1 + \gt x_2$$
for $d = {\gr}_1 ({\gl}_2 0 + {\gr}_2 0)$ and appropriate $\gs,\gt$ depending on n. Therefore:
$${\gl}_1 (z + w) = {\gs}{\gl}_2^{-1} z + T{\gt}{\gr}_2^{-1} w$$
(where $T x = x - d$) and ${\gl}_1 \in Hol(B;+)$\/.

Analogously we get ${\gr}_1, {\gl}_2, {\gr}_2 \in Hol(B;+)$\/.

Using Lemma \ref{hol2} we easily get (\ref{sol4}) for $i = 1,2$ where ${\ga}_i, {\gb}_i$ are automorphisms of $(B;+)$\/.

Replace $f_1$ and $f_2$ in (4.j):
$${\ga}_1({\ga}_2 x_1 + c_2 + {\gb}_2 x_2) + c_1 + {\gb}_1({\ga}_2 x_3 + c_2 + {\gb}_2 x_4) =$$
$$= {\ga}_2({\ga}_1 x_5 + c_1 + {\gb}_1 x_6) + c_2 + {\gb}_2({\ga}_1 x_7 + c_1 + {\gb}_1 x_8).$$
Replacing $x_1 = x_2 = x_3 = x_4 = 0$\/, we get:
$${\ga}_1 c_2 + c_1 + {\gb}_1 c_2 = {\ga}_2 c_1 + c_2 + {\gb}_2 c_1$$
i.e. $f_1(c_2,c_2) = f_2(c_1,c_1)$\/.

For $x_2 = x_3 = x_4 = 0$\/, we get:
$$\Lbranch(x_1,(4.\rm{j})) = {\ga}_1{\ga}_2 x_1 = \gg \gd x_1 = \Rbranch(x_1,(4.\rm{j}))$$
for some $\gg, \gd \in \{{\ga}_1, {\gb}_1, {\ga}_2, {\gb}_2 \}$ depending on j.

Analogously:
$$\Lbranch(x_i,(4.\rm{j})) = \Rbranch(x_i,(4.\rm{j}))$$
for $i = 2, 3, 4$\/.

The uniqueness of the group $(B;+)$ follows from the Albert
Theorem (see \cite{7}): If  two groups are isotopic, then
they are isomorphic.
\end{proof}
%
%\newpage
\section{Equations with Krsti\'c graph $P_{3}$}

There exist $32$ parastrophically uncancellable non-gemini and non-balanced quadratic functional equations with four object variables and two operations:
\begin{align}
\label{3-1}f_1(f_2(x,y),f_2(x,u))=f_2(f_1(y,v),f_1(u,v))\\
\label{3-2}f_1(f_2(x,y),f_2(x,u))=f_2(f_1(y,v),f_1(v,u))\\
\label{3-3}f_1(f_2(x,y),f_2(x,u))=f_2(f_1(u,v),f_1(y,v))\\
\label{3-4}f_1(f_2(x,y),f_2(x,u))=f_2(f_1(u,v),f_1(v,y))\\
\label{3-5}f_1(f_2(x,y),f_2(x,u))=f_2(f_1(v,y),f_1(u,v))\\
\label{3-6}f_1(f_2(x,y),f_2(x,u))=f_2(f_1(v,y),f_1(v,u))\\
\label{3-7}f_1(f_2(x,y),f_2(x,u))=f_2(f_1(v,u),f_1(y,v))\\
\label{3-8}f_1(f_2(x,y),f_2(x,u))=f_2(f_1(v,u),f_1(v,y))\\
\label{3-9}f_1(f_2(x,y),f_2(y,u))=f_2(f_1(x,v),f_1(u,v))\\
\label{3-10}f_1(f_2(x,y),f_2(y,u))=f_2(f_1(x,v),f_1(v,u))\\
\label{3-11}f_1(f_2(x,y),f_2(y,u))=f_2(f_1(u,v),f_1(x,v))\\
\label{3-12}f_1(f_2(x,y),f_2(y,u))=f_2(f_1(u,v),f_1(v,x))\\
\label{3-13}f_1(f_2(x,y),f_2(y,u))=f_2(f_1(v,x),f_1(u,v))\\
\label{3-14}f_1(f_2(x,y),f_2(y,u))=f_2(f_1(v,x),f_1(v,u))\\
\label{3-15}f_1(f_2(x,y),f_2(y,u))=f_2(f_1(v,u),f_1(x,v))\\
\label{3-16}f_1(f_2(x,y),f_2(y,u))=f_2(f_1(v,u),f_1(v,x))\\
\label{3-17}f_1(f_2(x,y),f_2(u,x))=f_2(f_1(y,v),f_1(u,v))\\
\label{3-18}f_1(f_2(x,y),f_2(u,x))=f_2(f_1(y,v),f_1(v,u)) \\
%\end{align}
%\begin{align}
\label{3-19}f_1(f_2(x,y),f_2(u,x))=f_2(f_1(u,v),f_1(y,v))\\
\label{3-20}f_1(f_2(x,y),f_2(u,x))=f_2(f_1(u,v),f_1(v,y))\\
\label{3-21}f_1(f_2(x,y),f_2(u,x))=f_2(f_1(v,y),f_1(u,v))\\
\label{3-22}f_1(f_2(x,y),f_2(u,x))=f_2(f_1(v,y),f_1(v,u))\\
\label{3-23}f_1(f_2(x,y),f_2(u,x))=f_2(f_1(v,u),f_1(y,v))\\
\label{3-24}f_1(f_2(x,y),f_2(u,x))=f_2(f_1(v,u),f_1(v,y))\\
\label{3-25}f_1(f_2(x,y),f_2(u,y))=f_2(f_1(x,v),f_1(u,v))\\
\label{3-26}f_1(f_2(x,y),f_2(u,y))=f_2(f_1(x,v),f_1(v,u))\\
\label{3-27}f_1(f_2(x,y),f_2(u,y))=f_2(f_1(u,v),f_1(x,v))\\
\label{3-28}f_1(f_2(x,y),f_2(u,y))=f_2(f_1(u,v),f_1(v,x))\\
\label{3-29}f_1(f_2(x,y),f_2(u,y))=f_2(f_1(v,x),f_1(u,v))\\
\label{3-30}f_1(f_2(x,y),f_2(u,y))=f_2(f_1(v,x),f_1(v,u))\\
\label{3-31}f_1(f_2(x,y),f_2(u,y))=f_2(f_1(v,u),f_1(x,v))\\
\label{3-32}f_1(f_2(x,y),f_2(u,y))=f_2(f_1(v,u),f_1(v,x))
\end{align}

The next lemma gives a general solution of the equation (\ref{3-10}) which generalizes the 
\emph{intermedial equation}
(see equation (4.36) and Theorem 8.4 of \cite{Krapez2007} for the original definition of intermedial equation).
\begin{lem}
\label{intermedial}
A general solution of the equation \emph{(\ref{3-10})} is given by:
\eql{sol51}{f_i(x,y) = {\ga}_i x + c_i + {\gb}_i y \quad (i = 1,2)}
where:
\begin{itemize}
\item $(B;+)$ is an arbitrary group,
\item $c_1, c_2$ are arbitrary elements of $B$ such that $f_1(c_2,c_2) = f_2(c_1,c_1)$\/,
\item ${\ga}_i, {\gb}_i \; (i = 1,2)$ are arbitrary automorphisms of + such that:
        \eql{lin51}{\Lbranch(z, (\ref{3-10})) = \Rbranch(z, (\ref{3-10}))}
        %for all linear variables $z$ of the equation \emph{(\ref{3-10})} and
         for $z \in \{ x, u \}$ and
         \eql{quadr51}{\Lbranch(w_i, (\ref{3-10})) w_i + c_i + \Rbranch(w_i, (\ref{3-10})) w_i = c_i}
         for $i \in \{ 1,2 \}, \/\/ w_1 = y$ and $w_2 = v$\/.
\end{itemize}
The group $(B;+)$ is unique up to isomorphism.
\end{lem}
\begin{proof}
(1)
To show that the pair $(f_1,f_2)$ of operations is a solution of (\ref{3-10}), just replace $f_i(x,y)$ in (\ref{3-10}) using (\ref{sol51}) and all conditions (\ref{lin51}), (\ref{quadr51}).

(2)
The equation (\ref{3-10}) is an instance of the generalized intermedial equation: 
\eqt{GI}{f_1(h_1(x,y),h_2(y,u))=f_2(h_3(x,v),h_4(v,u)).}
 
Choose $v = a$ for some $a \in B$ and define $\gg x = h_1(x,a), \gd u = h_2(a,u)$ and $g(x,u) = f_2(\gg x, \gd u)$\/. We get:
\eqt{GT}{f_1(h_1(x,y),h_2(y,u))=g(x,u)}
which is the generalized transitivity equation. By the Theorem \ref{krapez theorem}
all operations of this equation are isotopic to a group $+$ and the main operations $f_1, g$ can be chosen to be principally isotopic to it:
$$f_1(x,y) = {\gl}_1 x + {\gr}_1 y \qquad
g(x,y) = {\gl}_3 x + {\gr}_3 y.$$
It follows that $f_2(x,y) = {\gl}_3 {\gg}^{-1} x + {\gr}_3 {\gd}^{-1} y
= {\gl}_2 x + {\gr}_2 y$ for appropriate ${\gl}_2, {\gr}_2$\/.
Replacing this in (\ref{3-10}) we get:
\eql{im2}{{\gl}_1 ({\gl}_2 x + {\gr}_2 y) + {\gr}_1 ({\gl}_2 y + {\gr}_2 u) = {\gl}_2 ({\gl}_1 x + {\gr}_1 v) + {\gr}_2 ({\gl}_1 v + {\gr}_1 u).}

If we choose ${\gr}_2 u = {\gr}_1 v = 0$ and define 
$d = {\gr}_2 ({\gl}_1 {\gr}_1^{-1} 0 + {\gr}_1 {\gr}_2^{-1} 0)$ we get:
$${\gl}_1 ({\gl}_2 x + {\gr}_2 y) + {\gr}_1 {\gl}_2 y = {\gl}_2 {\gl}_1 x + d$$
which implies that ${\gl}_1 \in Hol(B;+)$\/. 

Analogously we get ${\gr}_1, {\gl}_2, {\gr}_2 \in Hol(B;+)$\/.

Using Lemma \ref{hol2} we easily get (\ref{sol51}) for $i = 1,2$ where ${\ga}_i, {\gb}_i$ are automorphisms of $(B;+)$\/.

Replace $f_1$ and $f_2$ in (\ref{3-10}):
$${\ga}_1({\ga}_2 x + c_2 + {\gb}_2 y) + c_1 + {\gb}_1({\ga}_2 y + c_2 + {\gb}_2 u) =$$
$$= {\ga}_2({\ga}_1 x + c_1 + {\gb}_1 v) + c_2 + {\gb}_2({\ga}_1 v + c_1 + {\gb}_1 u).$$
Putting $x = y = u = v = 0$\/, we get:
$${\ga}_1 c_2 + c_1 + {\gb}_1 c_2 = {\ga}_2 c_1 + c_2 + {\gb}_2 c_1$$
i.e. $f_1(c_2,c_2) = f_2(c_1,c_1)$\/.

For $y = u = v = 0$ we get:
$$\Lbranch(x,(\ref{3-10})) = {\ga}_1{\ga}_2 = {\ga}_2 {\ga}_1 = \Rbranch(x,(\ref{3-10})).$$

Analogously:
$$\Lbranch(u,(\ref{3-10})) = \Rbranch(u,(\ref{3-10}))\/,$$ 
$$\Lbranch(y,(\ref{3-10})) y + c_1 + \Rbranch(y,(\ref{3-10})) y =
{\ga}_1{\gb}_2 y + c_1 + {\gb}_1 {\ga}_2 y = c_1,$$
$$\Lbranch(v,(\ref{3-10})) v + c_2 + \Rbranch(v,(\ref{3-10})) v =
{\ga}_2{\gb}_1 v + c_2 + {\gb}_2 {\ga}_1 v = c_2.$$

The uniqueness of the group $(B;+)$ follows from the Albert
Theorem.
\end{proof}

\begin{lem}
\label{15a}
A general solution of the equation \emph{(5.j) \/ (j = 1,2,5,6,9,13,14,17,18, \\ 21,22,25,26,29,30)} is given by:
\eql{sol52}{f_i(x,y) = {\ga}_i x + c_i + {\gb}_i y \quad (i = 1,2)}
where:
\begin{itemize}
\item $(B;+)$ is an arbitrary Abelian group,
\item $c_1, c_2$ are arbitrary elements of $B$ such that $f_1(c_2,c_2) = f_2(c_1,c_1)$\/,
\item ${\ga}_i, {\gb}_i \; (i = 1,2)$ are arbitrary automorphisms of + such that:
        \eql{lin52}{\Lbranch(z, (5.\emph{j})) = \Rbranch(z, (5,\emph{j}))}
         for all linear variables $z$ of \emph{(5.j)} and
         \eql{quadr52}{\Lbranch(w, (\emph{5.j})) w + \Rbranch(w, (\emph{5.j})) w = 0} 
         for all quadratic variables $w$ from the equation.
\end{itemize}
The group $(B;+)$ is unique up to isomorphism.
\end{lem}
\begin{proof}
(1)
To show that the pair $(f_1,f_2)$ of operations is a solution of (5.j), just replace $f_i(x,y)$ in (5.j) using (\ref{sol52}) and all conditions (\ref{lin52}), (\ref{quadr52}).

(2) The crucial property of all 15 equations (5.j) is that, by applying duality to some of non--main operations of the generalized version of (5.j),
they may be transformed into equation (GI):
$$f_1(h_1(x,y),h_2(y,u))=f_2(h_3(x,v),h_4(v,u))$$
which, by the proof of the Lemma \ref{intermedial}, has a solution:
$$f_i(x,y) = {\ga}_i x + c_i + {\gb}_i y \quad (i = 1,2)$$
where $(B;+)$ is a group and ${\ga}_i, {\gb}_i$ are automorphisms of $+$\/.

Replacing $f_1, f_2$ in (5.j), we get:
\begin{equation}
\label{eq+}
\begin{aligned}
&&{\ga}_1({\ga}_2 x_1 + c_2 + {\gb}_2 x_2) + c_1 + {\gb}_1({\ga}_2 x_3 + c_2 + {\gb}_2 x_4) = \\
&&= \,{\ga}_2({\ga}_1 x_5 + c_1 + {\gb}_1 x_6) + c_2 + {\gb}_2({\ga}_1 x_7 + c_1 + {\gb}_1 x_8).
\end{aligned}
\end{equation}

Just as in the proof of Lemma \ref{intermedial}, we conclude that $f_1(c_2,c_2) = f_2(c_1,c_1)$\/. Let us define $c = f_1(c_2,c_2)$\/. 

To prove the properties from the statement of the Lemma, we need to discuss the arrangement 
$x_1 \dots x_4 = x_5 \dots x_8$ of variables in the equation (\ref{eq+}). It is easy to see:
\begin{itemize}
\item The order of first (i.e. left) appearances of variables is always $xyuv$\/.
\item $x_1 = x$\/.
\item Since $P_3$ has no nodes with loops, $x_2 = y$\/.
\item Either $x$ or $y$ is quadratic, but not both.
\item Variable $u$ is always linear.
\item Variable $v$ is always quadratic.
\item Arrangement $xyyu = xvvu$ is not allowed.
\end{itemize}

There are two possibilities: $x$ is either linear or quadratic.
\begin{itemize}
\item[a)] Variable $x$ is linear (and $y$ is quadratic). \\
Again, there are two possibilities: Either $x_3 = y$ or $x_3 = u$\/.
	\begin{itemize}
	\item[a1)] $x_3 = y$ (and $x_4 = u)$\/. \\
	Yet again, there are two possibilities: Either $x_5 = x$ or $x_5 = v$\/.
		\begin{itemize}
		\item[a11)] The arrangement of variables is $xyyu = xvuv$\/. \\
		We have equation (\ref{3-9}). Replacing $x = y = 0$ in (\ref{eq+}), we get:
		\eql{just2}{c + {\gb}_1 {\gb}_2 u = {\ga}_2 c_1 + {\ga}_2 {\gb}_1 v + c_2 + {\gb}_2 {\ga}_1 u + {\gb}_2 c_1 + {\gb}_2{\gb}_1v.}
		 For $v = 0$\, we get: 
		\eql{just u}{{\gb}_2 c_1 + {\gb}_1 {\gb}_2 u = {\gb}_2 {\ga}_1 u + {\gb}_2 c_1}
		and for $u = 0$\/: 
		\eql{just v}{c - {\gb}_2 {\gb}_1 v = {\ga}_2 c_1 + {\ga}_2 {\gb}_1 v + c_2 + {\gb}_2 c_1.}	
		Applying (\ref{just u}) and (\ref{just v}) to (\ref{just2}), we conclude:
		$$c + {\gb}_1 {\gb}_2 u - {\gb}_2 {\gb}_1 v = c - {\gb}_2 {\gb}_1 v + {\gb}_1 {\gb}_2 u$$
		which is, after cancellation from the left, equivalent to commutativity of $+$\/. Therefore $(B;+)$ is an Abelian 			group.
		\item[a12)] The arrangement of variables is $xyyu = vx(uv \text{ or } vu)$\/. \\
		Replacement $y = u = 0$ leads to:
		\eql{xv}{{\ga}_1 {\ga}_2 x + c = {\ga}_2 {\ga}_1 v + {\ga}_2 c_1 + {\ga}_2{\gb}_1 x + c_2 + t(v)}
		where  
		\begin{equation*}
		t(v) = \begin{cases}
			    {\gb}_2{\ga}_1 v + {\gb}_2 c_1, &\text{ if } x_7 = v \\
			    {\gb}_2 c_1 + {\gb}_2{\gb}_1 v, &\text{ if } x_7 = u
			    \end{cases}
		\end{equation*}
		Note that in both cases $t(0) = {\gb}_2 c_1$\/.
		Putting $x = 0$\/, we get:
		\eql{v}{t(v) = - c_2 - {\ga}_2 c_1 - {\ga}_2{\ga}_1 v + c}
		while replacement $v = 0$ leads to:
		\eql{x2}{{\ga}_1 {\ga}_2 x + {\ga}_2 c_1 = {\ga}_2 c_1 + {\ga}_2 {\gb}_1 x.}  
		Using (\ref{v}) and (\ref{x2}) in (\ref{xv}), we conclude:
		$${\ga}_1 {\ga}_2 x + c = {\ga}_2{\ga}_1 v + {\ga}_1 {\ga}_2 x - {\ga}_2{\ga}_1 v + c$$
		which implies that the group $(B;+)$ is Abelian.
		\end{itemize}
	\item[a2)] $x_3 = u$ (and $x_4 = y)$\/. \\
	The arrangement of variables is $xyuy = (xv \text{ or }vx)(uv \text{ or } vu)$\/.
	Replacement $x = v = 0$ in (5.j) yields:
	\eql{yu}{{\ga}_1 c_2 + {\ga}_1{\gb}_2 y + c_1 + {\gb}_1{\ga}_2 u + {\gb}_1 c_2 + {\gb}_1{\gb}_2 y = t(u)}
	where 
	\begin{equation*}
		t(u) = \begin{cases}
			    {\ga}_2 c_1 + c_2 + {\gb}_2{\ga}_1 u + {\gb}_2 c_1, &\text{ if } x_7 = u \\
			    c + {\gb}_2{\gb}_1 u, &\text{ if } x_7 = v
			    \end{cases}
	\end{equation*}
	Note that in both cases $t(0) = c$\/. Putting $y = 0$ in (\ref{yu}), we get:
	\eql{11}{{\ga}_1 c_2 + c_1 + {\gb}_1{\ga}_2 u + {\gb}_1 c_2 = t(u)}
	while replacement $u = 0$ yields:
	\eql{12}{{\ga}_1 c_2 + {\ga}_1{\gb}_2 y + c_1 = c - {\gb}_1{\gb}_2 y - {\gb}_1 c_2.}
	Feeding (\ref{11}) and (\ref{12}) in (\ref{yu}), we get:
	$$c - {\gb}_1{\gb}_2 y - {\gb}_1 c_2 + {\gb}_1{\ga}_2 u + {\gb}_1 c_2 = {\ga}_1 c_2 + c_1 + {\gb}_1{\ga}_2 u + {\gb}_1 c_2 - {\gb}_1{\gb}_2 y$$
	which implies commutativity of $+$\/.
\end{itemize}
\item[b)] Variable $x$ is quadratic (and $y$ is linear). \\
The arrangement of variables is $xy(xu \text{ or }ux) = (yv \text{ or }vy)(uv \text{ or } vu)$\/.
Let $u = v = 0$\/. We have:
\eql{xy}{{\ga}_1{\ga}_2 x + {\ga}_1 c_2 + {\ga}_1{\gb}_2 y + c_1 + s(x) = t(y)}
where:
\begin{equation*}
		s(x) = \begin{cases}
			   {\gb}_1{\ga}_2 x + {\gb}_1 c_2, &\text{ if } x_3 = x \\
			    {\gb}_1 c_2 + {\gb}_1{\gb}_2 x, &\text{ if } x_3 = u
			    \end{cases}
\end{equation*}
\begin{equation*}
		t(y) = \begin{cases}
			    {\ga}_2 {\ga}_1 y + c, &\text{ if } x_5 = y \\
			    {\ga}_2 c_1 + {\ga}_2{\gb}_1 y + c_2 + {\gb}_2 c_1, &\text{ if } x_5 = v.
			    \end{cases}
	\end{equation*}
Note that $s(0) = {\gb}_1 c_2$ and $t(0) = c$\/. Specifying $x = 0$\/, we get:
\eql{xx}{{\ga}_1 c_2 + {\ga}_1{\gb}_2 y + c_1 + {\gb}_1 c_2 = t(y)}
while $y = 0$ yields:
\eql{yy}{c_1 + s(x) = - {\ga}_1 c_2 - {\ga}_1{\ga}_2 x + c.}
Feeding (\ref{xx}) and (\ref{yy}) into (\ref{xy}), we get:
$${\ga}_1{\ga}_2 x + {\ga}_1 c_2 + {\ga}_1{\gb}_2 y - {\ga}_1 c_2 - {\ga}_1{\ga}_2 x + {\ga}_1 c_2 = {\ga}_1 c_2 + {\ga}_1{\gb}_2 y$$
which implies that the group $(B;+)$ is Abelian.
\end{itemize}

Because of commutativity of $+$ and the condition for $c$\/, the equation (5.j) reduces to:
\begin{equation*}
\begin{aligned}
&&{\ga}_1{\ga}_2 x_1 + {\ga}_1{\gb}_2 x_2 + {\gb}_1{\ga}_2 x_3 + {\gb}_1{\gb}_2 x_4 = \\
&&= \,{\ga}_2{\ga}_1 x_5 + {\ga}_2{\gb}_1 x_6 + {\gb}_2{\ga}_1 x_7 + {\gb}_2{\gb}_1 x_8).
\end{aligned}
\end{equation*}
which is equivalent to the system:
\begin{equation*}
\begin{cases}
\enspace \Lbranch(z, (5.\rm{j})) = \Rbranch(z, (5,\text{j})) \\
\enspace \Lbranch(w, (\rm{5.j})) w + \Rbranch(w, (\rm{5.j})) w = 0 
\end{cases}
\end{equation*}
         for all linear variables $z$ and all quadratic variables $w$\/. 

The uniqueness of the group $(B;+)$ follows from the Albert
Theorem.
\end{proof}

\begin{lem}
\label{extramedial}
A general solution of the equation \emph{(\ref{3-23})} is given by:
\begin{equation*}
\begin{cases}
\begin{aligned}
f_1(x,y) &= &\!\!{\ga}_1 x + c_1 + {\gb}_1 y  \\
f_2(x,y) &= &\!\!{\gb}_2 y + c_2 + {\ga}_2 x
\end{aligned}
\end{cases}
\tag{23}
\end{equation*}
where:
\begin{itemize}
\item $(B;+)$ is an arbitrary group,
\item $c_1, c_2$ are arbitrary elements of $B$ such that $f_1(c_2,c_2) = f_2(c_1,c_1)$\/,
\item ${\ga}_i, {\gb}_i \; (i = 1,2)$ are arbitrary automorphisms of + such that:
        \eql{lin53}{\Lbranch(z, (\ref{3-23})) = \Rbranch(z, (\ref{3-23}))}
         for $z \in \{ y, u \}$\/,
	  \eql{q531}{\Lbranch(x, (\ref{3-23})) x + c_1 + \Rbranch(x, (\ref{3-23})) x = c_1}
	   \eql{q532}{\Rbranch(v, (\ref{3-23})) v + c_2 + \Lbranch(v, (\ref{3-23})) v = c_2.}
\end{itemize}
The group $(B;+)$ is unique up to isomorphism.
\end{lem}
\begin{proof}
(1)
To show that the pair $(f_1,f_2)$ of operations is a solution of (\ref{3-23}), just replace $f_i(x,y)$ in (\ref{3-23}) using $(23)$ and all conditions (\ref{lin53})--(\ref{q532}).

(2)
Define new quasigroup $f_3$ to be the dual quasigroup of $f_2$\/, i.e. $f_3(x,y) = f_2(y,x)$\/. The equation (\ref{3-23}) transforms into equation (\ref{3-10}) with a general solution given by the Theorem \ref{intermedial}:
\begin{equation*}
\begin{cases}
\begin{aligned}
f_1(x,y) &= &\!\!{\ga}_1 x + c_1 + {\gb}_1 y  \\
f_3(x,y) &= &\!\!{\ga}_3 x + c_3 + {\gb}_3 y
\end{aligned}
\end{cases}
\tag{23*}
\end{equation*}
where:
\begin{itemize}
\item $(B;+)$ is an arbitrary group,
\item $c_1, c_3$ are arbitrary elements of $B$ such that $f_1(c_3,c_3) = f_3(c_1,c_1)$\/,
\item ${\ga}_i, {\gb}_i \; (i = 1,3)$ are arbitrary automorphisms of + such that:
$${\ga}_1{\ga}_3 = {\ga}_3{\ga}_1$$
$${\gb}_1{\gb}_3 = {\gb}_3{\gb}_1$$
$${\ga}_1{\gb}_3 x + c_1 + {\gb}_1{\gb}_3 x = c_1$$
$${\ga}_3{\gb}_1 v + c_3 + {\gb}_3{\ga}_1 v = c_3.$$
\end{itemize}

Define: ${\ga}_2 = {\gb}_3, {\gb}_2 = {\ga}_3$ and $c_2 = c_3$ and replace in (23*) to get: \\
$f_2(x,y) = f_3(y,x) = {\ga}_3 y + c_2 + {\gb}_3 x = {\gb}_2 y + c_2 + {\ga}_2 x$\/, and
$${\ga}_1{\gb}_2 = {\gb}_2{\ga}_1$$
$${\gb}_1{\ga}_2 = {\ga}_2{\gb}_1$$
$${\ga}_1{\ga}_2 x + c_1 + {\gb}_1{\ga}_2 x = c_1$$
$${\gb}_2{\gb}_1 v + c_2 + {\ga}_2{\ga}_1 v = c_2.$$
which is:
        $$\Lbranch(z, (5.23)) = \Rbranch(z, (5.23))$$
         for $z \in \{ y, u \}$\/, and
        $$\Lbranch(x, (5.23)) x + c_1 + \Rbranch(x, (5.23)) x = c_1,$$
	  $$\Rbranch(v, (5.23)) v + c_2 + \Lbranch(v, (5.23)) v = c_2.$$

Trivially, $f_1(c_2,c_2) = f_2(c_1,c_1)$\/.

The uniqueness of the group $(B;+)$ follows from the Albert
Theorem.
\end{proof}

\begin{lem}
\label{15b}
A general solution of the equation \emph{(5.k) \/ (k = 3,4,7,8,11,12,15,16,19, \\ 20,24,27,28,31,32)} is given by:
\eql{sol30}{f_i(x,y) = {\ga}_i x + c_i + {\gb}_i y \quad (i = 1,2)}
where:
\begin{itemize}
\item $(B;+)$ is an arbitrary Abelian group,
\item $c_1, c_2$ are arbitrary elements of $B$ such that $f_1(c_2,c_2) = f_2(c_1,c_1)$\/,
\item ${\ga}_i, {\gb}_i \; (i = 1,2)$ are arbitrary automorphisms of + such that:
        \eql{lin30}{\Lbranch(z, (5.\emph{k})) = \Rbranch(z, (5,\emph{k}))}
         for all linear variables $z$ of \emph{(5.k)} and
         \eql{q30}{\Lbranch(w, (\emph{5.k})) w + \Rbranch(w, (\emph{5.k})) w = 0} 
         for all quadratic variables $w$ from the equation.
\end{itemize}
The group $(B;+)$ is unique up to isomorphism.
\end{lem}
\begin{proof}
(1)
To show that the pair $(f_1,f_2)$ of operations is a solution of (5.k), just replace $f_i(x,y)$ in (5.k) using (\ref{sol30}) and all conditions (\ref{lin30}), (\ref{q30}).

(2)
Let us prove that the solution given in the Lemma is general in the case $k = 3$\/.

The equation (\ref{3-3}) has arrangement of variables equal to $xyxu = uvyv$\/. Let us replace the operation $f_2$ in (\ref{3-3}) by the dual operation $f_3(x,y) = f_2^*(x,y) = f_2(y,x)$\/. We get the equation 
$${f_1(f_3(y,x),f_3(u,x))=f_3(f_1(y,v),f_1(u,v))}$$
with the arrangement of variables equal to $yxux = yvuv$\/. Normalizing (i.e. applying the permutation $(xy)$ to variables) we get the equation (\ref{3-25}) with a general solution given in Lemma \ref{15a}:
\eql{q1}{f_i(x,y) = {\ga}_i x + c_i + {\gb}_i y \quad (i = 1,3)}
where:
\begin{itemize}
\item $(B;+)$ is an arbitrary group,
\item $c_1, c_3$ are arbitrary elements of $B$ such that $f_1(c_3,c_3) = f_3(c_1,c_1)$\/,
\item ${\ga}_i, {\gb}_i \; (i = 1,3)$ are arbitrary automorphisms of + such that:
 \eql{q2}{\Lbranch(z, (\ref{3-25})) = \Rbranch(z, (\ref{3-25}))}
         for all linear variables $z$ of (\ref{3-25}) and
         \eql{q3}{\Lbranch(w, (\ref{3-25})) w + \Rbranch(w, (\ref{3-25})) w = 0} 
         for all quadratic variables $w$ from the equation.
\end{itemize}
Conditions (\ref{q2}) and (\ref{q3}) evaluate to:
$${\ga}_1{\ga}_3 = {\ga}_3{\ga}_1$$
$${\gb}_1{\ga}_3 = {\gb}_3{\ga}_1$$
$${\ga}_1{\gb}_3 x + {\gb}_1{\gb}_3 x = 0$$
$${\ga}_3{\gb}_1 v + {\gb}_3{\gb}_1 v = 0.$$

Define: ${\ga}_2 = {\gb}_3, {\gb}_2 = {\ga}_3, c_2 = c_3$ and replace in (\ref{q1}) to get: \\
$f_2(x,y) = f_3(y,x) = {\ga}_3 y + c_3 + {\gb}_3 x = {\gb}_2 y + c_2 + {\ga}_2 x = {\ga}_2 x + c_2 + {\gb}_2 y$, and
$${\ga}_1{\gb}_2 = {\gb}_2{\ga}_1$$
$${\gb}_1{\gb}_2 = {\ga}_2{\ga}_1$$
$${\ga}_1{\ga}_2 x + {\gb}_1{\ga}_2 x = 0$$
$${\gb}_2{\gb}_1 v + {\ga}_2{\gb}_1 v = 0.$$
which is:
$$\Lbranch(z,(5,3)) = \Rbranch(z,(5,3))$$
 for $z \in \{ y, u \}$\/, and
        $$\Lbranch(w, (5.3)) w + \Rbranch(x, (5.3)) w = 0,$$
for	  $w \in \{x,v\}$\/.

Trivially, $f_1(c_2,c_2) = f_2(c_1,c_1)$\/.

Analogously, we can transform (5.4) into (5.29),
(5.7) into (5.26),
(5.8) into (5.30),
(5.11) into (5.17),
(5.12) into (5.21),
(5.15) into (5.18),
(5.16) into (5.22),
(5.19) into (5.9),
(5.20) into (5.13),
(5.24) into (5.14),
(5.27) into (5.1),
(5.28) into (5.5),
(5.31) into (5.2),
(5.32) into (5.6)
and prove appropriate relationships between ${\ga}_i,{\gb}_i,c_i \, (i = 1,2)$ for these equations, using results given in the Theorem \ref{15a}.
\end{proof}

\begin{comment}
\bf{The proof should be analogous to the proof of Theorem \ref{15a} via Equation (5.23) and Theorem \ref{extramedial}. Because of commutativity of $+$\/, the problem with $Lbaranch$ and $Rbranch$ should desapear.}
\end{comment}
%

\begin{defn}
Let ${\partial}$ be operator acting on terms of the group $(B;+)$ so that 
\begin{equation*}
\partial(t) =
		\begin{cases}
             t, \enspace\qquad \text{ if } t \text{ is a monomial}, \\
             t_2 + t_1,   \text{ if } t = t_1 + t_2\/.
             \end{cases}
\end{equation*}
\end{defn}

It is easy to see that 
for all natural numbers $n$\/, $\partial(x_1 + x_2 + \cdots + x_n) = x_n + x_{n-1} + \cdots x_1$\/. In particular 
$\partial(x + y + z) = z + y + x$\/. Also,
for all even (odd) $j$ and all terms $t$\/:  ${\partial}^j(t) = t \; ({\partial}^j(t) = \partial(t))$\/.

We may now combine Lemmas \ref{intermedial} and \ref{extramedial} into:
\begin{thm}
\label{2}
A general solution of the equation \emph{(5.j) \; (j = 10,23)} is given by:
\begin{equation*}
		\begin{cases}
			    f_1(x,y) = {\ga}_1 x + c_1 + {\gb}_1 y  \\
			    f_2(x,y) = {\partial}^{\emph{j}}({\ga}_2 x + c_2 + {\gb}_2 y)
  	    \end{cases}
\end{equation*}
where:
\begin{itemize}
\item $(B;+)$ is an arbitrary group,
\item $c_1, c_2$ are arbitrary elements of $B$ such that $f_1(c_2,c_2) = f_2(c_1,c_1)$\/,
\item ${\ga}_i, {\gb}_i \; (i = 1,2)$ are arbitrary automorphisms of + such that:
$$\Lbranch(z, \emph{(5.j)}) = \Rbranch(z, \emph{(5.j)})$$
        for all linear variables $z$ of the equation \emph{(5.j)} and
$$\Lbranch(w_i, \emph{(5.j)}) w_i + c_i + \Rbranch(w_i, \emph{(5.j)}) w_i = c_i$$
         for $i \in \{ 1,2 \},$ where $w_1$ is the left quadratic variable while $w_2$ is the right quadratic variable of \emph{(5.j)}.
\end{itemize}
The group $(B;+)$ is unique up to isomorphism.
\end{thm}
%
%\newpage
Likewise, the Theorem \ref{main1} and Lemmas \ref{15a} and \ref{15b} can be combined into:
\begin{thm}
\label{main2}
A general solution of the equation \emph{$(m.{\rm{j}_m})$ \; $(m = 4,5; \, 1 \leq \rm{j}_4 \leq 16; \, 1 \leq \rm{j}_5 \leq 32; \, \rm{j}_5 \neq 10,23)$} is given by:
$$f_i(x,y) = {\ga}_i x + c_i + {\gb}_i y \quad (i = 1,2)$$
where:
\begin{itemize}
\item $(B;+)$ is an arbitrary Abelian group,
\item $c_1, c_2$ are arbitrary elements of $B$ such that $f_1(c_2,c_2) = f_2(c_1,c_1)$\/,
\item ${\ga}_i, {\gb}_i \; (i = 1,2)$ are arbitrary automorphisms of + such that:
        $$\Lbranch(z, (m.\rm{j}_m)) = \Rbranch(z, (m.\rm{j}_m))$$
         for all linear variables $z$ of $(m.\rm{j}_m)$ and
         $$\Lbranch(w, (m.\rm{j}_m)) w + \Rbranch(w, (m.\rm{j}_m)) w = 0$$
         for all quadratic variables $w$ from the equation.
\end{itemize}
The group $(B;+)$ is unique up to isomorphism.
\end{thm}
%

%%%%%%%%%%%%%%%%%%%%%%%%%%%%%%%%%%%%%%%%%%%%%%%%%%%%%%%%%%%%%%%%%

\section{Algebras with Parastrophically Uncancellable Quadratic Hyperidentities}

By \cite{Movsisyan1998, Movsisyan2001}, a hyperidentity (or $\forall (\forall)$-identity) is a second-order formula of the following form:
\[\forall f_1,\ldots,f_k \forall x_1,\ldots,x_n \qquad (w_1=w_2),\]
where $w_1$, $w_2$ are words (terms) in the alphabet of function variables $f_1, \ldots, f_k$ and object variables $x_1, \ldots, x_n$. However hyperidetities are usually presented without universal quantifiers: $w_1=w_2$. The hyperidentity $w_1=w_2$ is said to be satisfied in the algebra $(B; F)$ if this equality holds whenever every function variable $f_i$ is replaced by an arbitrary operation of the corresponding arity from $F$ and every object variable $x_i$ is replaced by an arbitrary element of $B$. 

Now, as a consequence of the results of the previous section, we can establish the following representation of a binary algebra satisfying one of the non-gemini hyperidentities.

\begin{thm}
\label{hyper}
Let $(B; F)$ be a binary  algebra with quasigroup
operations which satisfy one of the non--gemini hyperidentities \emph{$(m.{\rm{j}_m})$ \; $(m = 4,5; \, 1 \leq \rm{j}_4 \leq 16; \, 1 \leq \rm{j}_5 \leq 32)$}. 
%with the Krsti\'c graphs $K_{3,3}$ and $P_3$ be given, 
Then there exists an Abelian group $(B;+)$ such that
every operation $f_i\in F$ is represented by:
\[
f_i(x,y)=\alpha_i(x)+c_i+\beta_i(y),
\]
where:
\begin{itemize}
\item $c_i \; (i = 1,\ldots, |F|)$ are arbitrary elements of $B$ such that $f_l(c_k,c_k) = f_k(c_l,c_l)$\/ for $1 \leqslant l,k \leqslant |F|$\/,
\item ${\ga}_i, {\gb}_i \; (i = 1,\ldots, |F|)$ are arbitrary automorphisms of + such that:
        $$\Lbranch(z, (m.\rm{j}_m)) = \Rbranch(z, (m.\rm{j}_m))$$
         for all linear variables $z$ of $(m.\rm{j}_m)$ and
         $$\Lbranch(w, (m.\rm{j}_m)) w + \Rbranch(w, (m.\rm{j}_m)) w = 0$$
         for all quadratic variables $w$ from the equation.
\end{itemize}
\end{thm}
\begin{proof}
Let us consider the pair $(f_1,f_1)$ of operations satisfying equation $(m.{\rm{j}_m})$ (for $m = 4$ or 5; $\rm{j}_4$ is some of $1,2,\dots,16$ while $\rm{j}_5$  is some of $1,2,\dots,32$). Then
$$f_1(x,y) = {\ga}_1(x) + c_1 + {\gb}_1(y)$$
where $+$ is a group and ${\ga}_1,{\gb}_1$ its automorphisms. In case of the equation (5.10) ((5.23)) the group $+$ is commutative by Theorem \ref{quasimedial} (Theorem \ref{quasipara}). In all other cases $+$ is commutative by Theorem \ref{main2}.

For any $i \in F, \; i \neq 1$\/, the pair $(f_1,f_i)$ also satisfies $(m.{\rm{j}_m})$\/, hence both are principally isotopic to a group (perhaps other than $+$). Anyway, $f_i$ is also principally isotopic to $+$ and by Theorem \ref{2} or \ref{main2}
$$f_i(x,y) = {\ga}_i(x) + c_i + {\gb}_i(y)$$
where $c_i\in B$ and $\alpha_i,\beta_i\in Aut(B;+)$ such that
 $$\Lbranch(z, (m.\rm{j}_m)) = \Rbranch(z, (m.\rm{j}_m))$$
         for all linear variables $z$ of $(m.j_m)$ and
         $$\Lbranch(w, (m.\rm{j}_m)) w + \Rbranch(w, (m.\rm{j}_m)) w = 0$$
         for all quadratic variables $w$ from the equation.

The rest of the proof is easy.
\end{proof}
%

%%%%%%%%%%%%%%%%%%%%%%%%%%%%%%%%%%%%%%%%%%%%%%%%%%%%%%%%%%%%%%%%%%%%

\section*{Acknowledgement}

The work of A. Krape\v{z} is supported by the Ministry of Education, Science and Technological Development of the Republic of Serbia, grants ON 174008 and ON 174026.

The work of Yu. Movsisyan is supported by the 'Center of Mathematical Research' of the state Committee of Science of the Republic of Armenia.

%%%%%%%%%%%%%%%%%%%%%%%%%%%%%%%%%%%%%%%%%%%%%%%%%%%%%%%%%%%%%%

%%%%%%%%%%%%%%%%%%%%%%%%%%%%%%%%%%%%%%%%%%%%%%%%%%%%%%%%%%%%

\end{document}